\documentclass{article}
\usepackage{amsthm,amsmath,amsfonts,amssymb}

\title{Regularity theory for nonlinear integral operators}
\author{Luis Caffarelli, Chi Hin Chan, Alexis Vasseur}

\newlength{\hchng}
\newlength{\vchng}
\newtheorem{thm}{Theorem}[section]

\newtheorem{cor}[thm]{Corollary}
\newtheorem{lemma}[thm]{Lemma}

\newtheorem{preremark}[thm]{Remark}

\numberwithin{equation}{section}

\newcommand{\var}{w}
\newcommand{\pha}{F}
\newcommand{\A}{\mathcal{A}}
\newcommand{\ds}{\displaystyle}

\newcommand{\R}{\mathbb R}
\newcommand{\eps}{\varepsilon}

\begin{document}
\maketitle
\bibliographystyle{plain}
\noindent{\bf Abstract:} 
This article is dedicated to the proof of the existence of classical solutions for a class of non-linear integral variational problems. Those problems are involved
in nonlocal image and signal processing.



\section{Introduction}

The purpose of this work is to develop a regularity theory for non-local evolution equations of variational type with "measurable" kernels. More precisely, we consider solutions of the evolution equations of the type
\begin{equation}\label{star}
w_t(t,x)=\int[w(t,y)-w(t,x)]K(t,x,y)\,dy,
\end{equation}
where all that is required of the kernel $K$ is that there exists  $0<s<1$ and $0<\Lambda$, such that 
\begin{equation}\label{diese}
\begin{array}{l}
\mathrm{symmetry  \  in} \  x,y:\qquad  K(t,x,y)=K(t,y,x) \  \mathrm{for \  any \ } x\neq y,\\[0.3cm]
\ds{{\mathbf 1}_{\{|x-y|\leq 3\}}\frac{(1-s/2)}{\Lambda}|x-y|^{-(N+s)}\leq K(t,x,y)\leq (1-s/2){\Lambda}|x-y|^{-(N+s)}.}
\end{array}
\end{equation}
The symmetry of the kernel $K$ makes of the operator: 
$$
\int[w(y)-w(x)]K(x,y)\,dy 
$$
the Euler Lagrange equation of the  energy integral
$$
E(w)=\int\int[w(x)-w(y)]^2K(x,y)\,dx\,dy.
$$
It suggests a mathematical treatment based on the De Giorgi-Nash-Moser ideas \cite{DeGiorgi,Nash} from the calculus of variations. In fact, one of the immediate applications of our result is to nonlinear variational integrals 
$$
E_{\phi}(w) =\int\int\phi(w(x)-w(y))K(x-y)\,dx\,dy,
$$
for $\phi$ a $C^2$ strictly convex functional. Indeed, the fact that $K(x,y)$ has the special form $K(x-y)$ makes the equation translation invariant, and as in the second order case, this implies that first derivatives of $w$ satisfy an equation of the type (\ref{star}). Our results are basically that solutions with initial data in $L^2$ become instantaneously bounded and Holder continuous. In these lines, see the work of Kassmann \cite{Kassmann}, Kassmann and Bass \cite{BassKassmann} (see also 
 \cite{BassLevin} and \cite{nonlocalDirichletform}),
where the Moser approach for the stationary case is fully developped. 
 For the non divergence case there is a recent work of 
Silvestre (see \cite{Silvestre2006}, \cite{CaffarelliSilvestre}, and references therein). We were motivated by our work on Navier-Stokes \cite{vasseur}  and 
the quasigeostrophic equations \cite{Driftdiffusion}. In this work, the full regularity of the solutions to the surface quasi-geostrophic equation is shown in the critical case. It was followed by several works on the same subject in the super-critical case (see for instance \cite{ConstantinWu}). Note also that the result was obtained, using completely different techniques  by Kiselev, Nazarov and Volberg \cite{kiselev}.  Our approach led to some progress in the supercritical case (see  \cite{silvestre-2008,chan-2009}).  
It  follows pretty much the lines of the De Giorgi's work \cite{DeGiorgi}. 
Non linear equations of this form appear extensively in the phase transition literature (see Giacomin, Lebowitz, and Presutti \cite{Presutti})
and more recently on issues of image processing (Gilboa and Osher \cite{gilboa}). 

\section{Presentation of the results}

Consider the variational integral 
$$
V(\theta)=\int_{\mathbb{R}^{N}} \int_{\mathbb{R}^{N}} \phi (\theta (y) - \theta (x)) K (y-x) dy dx, 
$$
 for  $\phi : \mathbb{R} \rightarrow [0, \infty )$   an even convex function of class $C^{2} (\mathbb{R})$  satisfying the conditions 
 \begin{equation}\label{Hypothesis_Phi}
 \begin{array}{l}
\displaystyle{\phi (0) = 0},\\ 
\displaystyle{ \Lambda^{-1/2} \leq \phi '' (x)\leq \Lambda^{1/2} ,\qquad x\in\R,}
\end{array}
\end{equation}
for  a given constant $1<\Lambda<\infty$.

The kernel $K : \mathbb{R}^{N} - \{0\} \rightarrow (0, \infty )$ is supposed to satisfy the following  conditions for a $0<s<2$.
\begin{equation}\label{Hypothesis_K}
\begin{array}{l}
\displaystyle{K(-x) = K(x) , \qquad \mathrm{ for \  any \  } x \in \mathbb{R}^{N} - \{0\}, }\\[0.3cm]
\displaystyle{{\mathbf 1}_{\{|x\leq 3\}}(1-s/2)\frac{\Lambda^{-1/2}}{|x|^{N+s}} \leq K(x) \leq (1-s/2)\frac{\Lambda^{1/2}}{|x|^{N+s}},\qquad \mathrm {for \  any  \ } x \in \mathbb{R}^{N} - \{0\}. } 
\end{array}
\end{equation}
With the above setting, the corresponding Euler-Lagrange equation for the variational integral $\int_{\mathbb{R}^{N}} \int_{\mathbb{R}^{N}} \phi (\theta (y) - \theta (x)) K (y-x) dy dx $ is given by

\begin{equation*}
- \int_{\mathbb{R}^{N}} \phi ' (\theta (y) - \theta (x)) K(y-x) dy = 0 . 
\end{equation*}

We are considering in this paper the associated time dependent problem:
\begin{equation}\label{nonlinearparabolic}
\partial_{t} \theta (t,x) - \int_{\mathbb{R}^{N}} \phi ' (\theta (t,y) - \theta (t,x)) K(y-x) dy = 0.  
\end{equation}

The main goal of this paper is to address the regularity problem for solutions to the above parabolic-type equation and establish the following main theorem.

\begin{thm}\label{thm1}
Consider an even convex function $\phi$ verifying the Hypothesis (\ref{Hypothesis_Phi}) and a kernel $K$ verifying Hypothesis (\ref{Hypothesis_K}) for a $0<s<2$.
Then, for any  initial datum $\theta_{0} \in H^1(\mathbb{R}^{N})$ , there exists a global classical solution  to Equation \ref{nonlinearparabolic} with $\theta (0 , \cdot ) = \theta_{0}$ in the $L^{2}(\mathbb{R}^{N})$ sense. Moreover $\nabla_x\theta \in C^{\alpha}( (t_0,\infty ) \times \mathbb{R}^{N}) $ for any $t_0>0$.
\end{thm}

The existence of weak solutions with nonincreasing energy can be constructed following \cite{Fourrier}.
To address the regularity problem for solutions to Equation (\ref{nonlinearparabolic}), we follow the classical idea of De Giorgi and look at the first derivative $D \theta$ of a solution $\theta$ to Equation (\ref{nonlinearparabolic}). 
First, we use the change of variable $y = x + z$ to rewrite Equation (\ref{nonlinearparabolic}) as follows
\begin{equation}\label{eq_newversion}
\partial_{t} \theta  - \int_{\mathbb{R}^{N}} \phi ' (\theta ( x + z ) - \theta (x)) K(z) dz = 0 .
\end{equation}
Now, we consider $w=D_e\theta$, the derivative in the direction $e$ of $\theta$. Derivating (formally) Equation (\ref{nonlinearparabolic}) in the direction $e$ we find
\begin{equation*}
\partial_{t} w  - \int_{\mathbb{R}^{N}} \phi '' (\theta ( x + z ) - \theta (x)) \{ w (x + z) -w (x)   \}  K(z) dz = 0 . 
\end{equation*}

We then use the change of variable back to $y = x + z$  to rewrite the above equation in the following way
\begin{equation*}
\partial_{t} w - \int_{\mathbb{R}^{N}} \phi '' (\theta ( y ) - \theta (x)) \{ w (y) - w (x)   \}  K(y-x) dz = 0 . 
\end{equation*}

Consider the new kernel $K (t,x,y) = \phi '' (\theta (t, y ) - \theta (t,x)) K(y-x) $ (with an obvious slight abuse for notation). Since $\phi$ is an even function, $\phi ''$ is also an even function, and hence the new kernel $K (t,x, y)$ is symmetric in $x$ and $y$. Moreover, Hypothesis (\ref{Hypothesis_K}) and 
(\ref{Hypothesis_Phi}) implies  that $K (t,x,y)$ satisfies the condition 
$$
(1-s/2){\mathbf 1}_{\{|x-y|\leq 3\}}\frac{\Lambda^{-1}}{|x-y|^{N +s}} \leq K (t,x, y ) \leq (1-s/2)\frac{\Lambda}{|x-y|^{N+s}}.
$$ 
As a result, the function $w = D_e\theta$ satisfies  Equation (\ref{star})
with the kernel $K (t,x,y)$ verifying Hypothesis (\ref{diese}).
Our goal is then to show that solutions to Equation (\ref{star})  are in $C^\alpha$. 
\vskip0.3cm

To make the argument rigorous, we will consider the difference quotient 
$D_{e}^{h} \theta ( \cdot ) = \frac{1}{h} \{ \theta ( \cdot + h e ) - \theta (\cdot )      \} $. We use again the version (\ref{eq_newversion}) of Equation (\ref{nonlinearparabolic}).
For any given $\eta \in C_{c}^{\infty}(\mathbb{R}^{N})$, we use the difference quotient $D_{e}^{-h} \eta $ to test against it, and we get
\begin{equation*}
 \int_{\mathbb{R}^{N}} \partial_{t} \theta (t,x) D_e^{-h} \eta (x) dx   -  \int_{\mathbb{R}^{N}} \int_{\mathbb{R}^{N}} \phi ' (\theta (t, x + z ) - \theta (t,x)) D_{e}^{-h} \eta (x) dx
K(z) dz = 0 .
\end{equation*}

Using the discrete integration by part $ \int_{\mathbb{R}^{N}} f(x) D_e^{-h} g (x) dx = - \int_{\mathbb{R}^{N}} D_e^{h}f (x) g(x) dx $, we find
\begin{equation*}
 \int_{\mathbb{R}^{N}} \partial_{t} D_e^{h} \theta (t,x) \cdot \eta (x) dx   -  \int_{\mathbb{R}^{N}} \int_{\mathbb{R}^{N}} D_e^{h}[\phi ' (\theta ( \cdot + z ) - \theta (\cdot ))](x) \cdot \eta (x) dx
K(z) dz = 0.
\end{equation*}
The change of variable $y = x + z$ leads to
\begin{equation*}\label{introducionequationtwo}
\int_{\mathbb{R}^{N}} \partial_{t} D_e^{h} \theta (t,x) \cdot \eta (x) dx   -  \int_{\mathbb{R}^{N}} \int_{\mathbb{R}^{N}} D_e^{h}[\phi ' (\theta ( \cdot + y-x ) - \theta (\cdot ))](x) \cdot \eta (x)  K(y-x) dx dy = 0 .
\end{equation*}
Note that  $\phi$ is an even function, so $\phi '$ is an odd function and consequently 
\begin{equation*}
D_e^{h}[\phi ' (\theta ( \cdot + y-x ) - \theta (\cdot ))](x)  = - D_e^{h}[\phi ' (\theta ( \cdot + x-y ) - \theta (\cdot ))](y) .
\end{equation*}
Using also the symmetry of $K$, we can symmetrize the operator to get
\begin{equation}\label{introducionequationthree}
\int_{\mathbb{R}^{N}} \partial_{t} D_e^{h} \theta (t,x) \cdot \eta (x) dx   - \frac{1}{2} \int_{\mathbb{R}^{N}} \int_{\mathbb{R}^{N}} D_e^{h}[\phi ' (\theta ( \cdot + y-x ) - \theta (\cdot ))](x) \cdot [\eta (x) - \eta (y)]  K(y-x) dx dy = 0 .
\end{equation}
Setting $Y=\theta(y+he)-\theta(x+he)$ and $X=\theta(y)-\theta(x)$, we get
\begin{equation*}
\begin{split}
&D_e^{h}[\phi ' (\theta ( \cdot + y-x ) - \theta (\cdot ))](x)  = \frac{1}{h} \{ \phi ' (  \theta (y + h e_{i}) - \theta (x + h e_{i})  ) - \phi ' ( \theta (y) - \theta (x)  ) \}
\\
&\qquad = \frac{Y-X}{h} \int_0^1\phi '' (X+s(Y-X))\,ds \\
&\qquad = [  D_e^{h}\theta (y)      -     D_e^{h}\theta (x)]      \int_0^1\phi''((1-s)[\theta(t,y)-\theta(t,x)]+s[\theta(t,y+he)-\theta(t,x+he)])\,ds.
\end{split}
\end{equation*}
Hence, 
$w = D_e^{h}\theta$  solves the following equation
\begin{equation*}
\int_{\mathbb{R}^{N}} \partial_{t}w (t,x) \eta (x) dx + \int_{\mathbb{R}^{N}} \int_{\mathbb{R}^{N}} K^h(t,x,y) [\eta (x) - \eta (y) ] [w (t,x) - w (t,y)] dy dx = 0 .
\end{equation*}
where 
$$
K^h(t,x,y)=K(y-x)\int_0^1\phi''((1-s)[\theta(t,y)-\theta(t,x)]+s[\theta(t,y+he)-\theta(t,x+he)])\,ds.
$$
Note that this new kernels verified independently on $h$ the properties (\ref{diese}) with the same $\Lambda$.

Theorem \ref{thm1} is then  a consequence of the following theorem.
\begin{thm}\label{thm2}
Let $w$ be a weak solution of (\ref{star}) with a kernel verifying the properties (\ref{diese}), then
for every $t_0>0$, $w\in C^\alpha((t_0,\infty)\times\R^N)$. The constant $\alpha$ and the norm of $w$  depend only on $t_0$, $N$, $\|w^0\|_{L^2}$, and
$\Lambda$.
\end{thm}
Passing into the limit $h\to0$ gives the result of Theorem \ref{thm1}.
The rest of the paper is dedicated to the proof of Theorem \ref{thm2}.

\section{The first De-Giorgi's lemma}
In this section and the next section, we focus on the  differential equation stated in the sense of weak formulation in (\ref{star}).
We rewrite it in the following way.
\begin{equation}\label{main}
\begin{aligned}
\int_{\mathbb{R}^{N}}\partial_{t}\var (t,x) \cdot \eta (x) dx  + B [\var (t,\cdot ) , \eta ] = 0 , \forall \eta \in C_{c}^{\infty }(\mathbb{R}^{N}),\\
B [u,v] = \int_{\mathbb{R}^{N}} \int_{\mathbb{R}^{N}}  K(t,x,y) [u(x) - u(y)]\cdot [v(x) - v(y)] dx dy ,
\end{aligned}
\end{equation}
where the kernel $K(t,x,y)$ is assumed to verify the Hypothesis (\ref{diese}).
We first  introduce the following   function $\psi$:
\begin{equation}\label{defofpsi}
\psi (x) = (|x|^{\frac{s}{2}} - 1)_+ .  
\end{equation}
 For any $L \geq 0$, we define 
 \begin{equation}\label{psiL}
 \psi_{L} (x) = L + \psi (x) .
 \end{equation}
With the above setting, the first De Giorgi's lemma is as follows.

\begin{lemma}\label{firstDeGiorgilemma}
Let $\Lambda$ be the given constant in condition (\ref{diese}). Then,  there exists a constant  $\epsilon_{0} \in (0, 1)$,  depending only on $N$, $s$, and $\Lambda$, such that for any solution $\var : [-2, 0] \times \mathbb{R}^{N} \rightarrow \mathbb{R}$ to (\ref{main}),   the following implication for $\var$ holds true.

 If it is verified that 
 $$
 \int_{-2}^{0} \int_{\mathbb{R}^{N}} [\var (t,x) - \psi (x)]_{+}^{2} dx dt \leq \epsilon_{0},
 $$
  then we have 
  $$
  \var (t,x) \leq  \frac{1}{2} + \psi (x)
  $$ 
  for $(t,x) \in [-1,0]\times \mathbb{R}^{N}$. (Hence, we have in particular that $\var \leq 1/2$ on $[-1, 0]\times B(1)$ .)
\end{lemma}

The main difficulty in our approach is due to  the nonlocal operator. In \cite{Driftdiffusion}, a localization of the problem was performed at the cost of adding one more variable to the problem. This was based on the ``Dirichlet to Neuman" map. This approach still works for any fractional Laplacian
(see Caffarelli and Silvestre \cite{Ext}). However it breaks down for general kernels as (\ref{diese}). Instead, we keep track of the far away behavior
of the solution via the function $\psi$. 
\vskip0.3cm
\noindent{Remark:} 
All the computations on weak solutions in the proof can be justified by replacing the variable kernel in a
neighborhood of the origin by the fractional Laplacian through a smooth
cut off. Then the equation becomes a fractional heat equation with a
bounded right hand side, thus $C^2$ in space.  This makes the integrals
involved uniformly convergent. Once the a priori Holder continuity is
proven, we pass to the limit. 
\vskip0.3cm

\begin{proof}

We split the proof in several steps. 
\vskip0.3cm
\noindent{\bf First step: Energy estimates.} Let $\var : [-2, 0] \times \mathbb{R}^{N} \rightarrow \mathbb{R}$ be a solution to equation (\ref{main}). For $0\leq L\leq 1$, we consider the truncated function $[\var - \psi_{L}]_{+}$, where $\psi_L$ is defined by (\ref{psiL}). Then, we take the test function $\eta$ to be $[\var - \psi_{L}]_{+}$ in the weak formulation of equation (\ref{main}), which gives

\begin{equation}\label{energie}
\begin{array}{l}
\ds{0  = \frac{1}{2} \frac{d}{dt} \int_{\mathbb{R}^{N}} [\var - \psi_{L}]_{+}^{2} dx + B[ \var   ,    (\var - \psi_{L})_{+}       ] }\\[0.3cm]
\qquad\ds{ =  \frac{1}{2} \frac{d}{dt} \int_{\mathbb{R}^{N}} [\var - \psi_{L}]_{+}^{2} dx   +    B [ (\var - \psi_{L})_{+} , (\var - \psi_{L})_{+}    ]  + B[ (\var - \psi_{L})_{-} , (\var - \psi_{L})_{+}       ]} \\[0.3cm]
\qquad\qquad \ds{+ B[ \psi_{L} , (\var - \psi_{L})_{+}      ] .}
\end{array}
\end{equation}

Now, due to the observation that $ (\var - \psi_{L})_{+} \cdot (\var - \psi_{L})_{-} = 0 $ and the symmetry of $K$ in $x,y$, we have

$$ 
 B[ (\var - \psi_{L})_{-} , (\var - \psi_{L})_{+}  ] = 2\int_{\mathbb{R}^{N}}  \int_{\mathbb{R}^{N}} K (t,x,y) (\var - \psi_{L})_{+}(x)  (\var - \psi_{L})_{neg} (y) dx dy ,
 $$
where we denote $ (\var - \psi_{L})_{neg} =  - (\var - \psi_{L})_{-} \geq 0 $ .  In particular  
$$
B[ (\var - \psi_{L})_{-} , (\var - \psi_{L})_{+}  ] \geq 0. 
$$
This ``good term" is not fully exploited in this section. It will be used in a crucial way in the next section. 
The remainder can be written as:
\begin{equation}\label{3in1stDeGiorgi}
\begin{split}
&B[ \psi_{L} , (\var - \psi_{L})_{+}  ]  \\
& = \frac{1}{2} \int \int_{|x-y| \geq 1} K(t,x,y)  [\psi_{L} (x) - \psi_{L}(y)] \cdot  \{ (\var - \psi_{L})_{+}(x) - (\var - \psi_{L})_{+}(y)  \}  dx dy \\
& + \frac{1}{2} \int \int_{|x-y| < 1} K(t,x,y)  [\psi_{L} (x) - \psi_{L}(y)] \cdot  \{ (\var - \psi_{L})_{+}(x) - (\var - \psi_{L})_{+}(y)  \}  dx dy .
\end{split}
\end{equation}
Using the  inequality $|\psi (x) - \psi (y)| \leq 2 |y-x|^{\frac{s}{2}}$, for any $x$ and $y$ with $|y-x| \geq 1$, we get the following estimation of the ``far-away" contribution.
\begin{align*}
&\left | \int \int_{|x-y| \geq 1} K(t,x,y)  [\psi_{L} (x) - \psi_{L}(y)] \cdot   [\var - \psi_{L}]_{+}(x)   dx dy \right| \\
& =\left | \int \int_{|x-y| \geq 1} K(t,x,y)  [\psi (x) - \psi (y)] \cdot   [\var - \psi_{L}]_{+}(x)   dx dy\right| \\
& \leq \int_{\mathbb{R}^{N}} \int_{|y-x| \geq 1} \frac{2\Lambda }{|x-y|^{N+s}} 2 |y-x|^{\frac{s}{2}} dy \cdot (\var - \psi_{L})_{+}(x) dx \\
& = 4\Lambda |S^{N-1}| \int_{1}^{\infty} r^{-\frac{s}{2}} dr \int_{\mathbb{R}^{N}} (\var - \psi_{L})_{+}(x) dx 
  \leq C  \int_{\mathbb{R}^{N}} (\var - \psi_{L})_{+}(x) dx .
\end{align*}
By symmetry we end up to
\begin{equation}\label{4in1stDeGiorgi}
\begin{split}
&\left |\int \int_{|x-y| \geq 1} K(t,x,y)  [\psi_{L} (x) - \psi_{L}(y)] \cdot  \{ (\var - \psi_{L})_{+}(x) - (\var - \psi_{L})_{+}(y)  \}  dx dy\right| \\ 
& \leq C \int_{\mathbb{R}^{N}} (\var - \psi_{L})_{+}(x) dx .
\end{split}
\end{equation}

The other part of the remainder ca be controlled in the following way:
\begin{equation}\label{5in1stDeGiorgi}
\begin{split}
& \left|\int \int_{|x-y| < 1} K(t,x,y)  [\psi_{L} (x) - \psi_{L}(y)] \cdot  \{ (\var - \psi_{L})_{+}(x) - (\var - \psi_{L})_{+}(y)  \}  dx dy\right| \\
& \leq 2 \int \int_{|x-y| < 1} K(t,x,y) \chi_{\{  [\var - \psi_{L}](x) > 0     \}}  |\psi_{L} (x) - \psi_{L}(y)| | (\var - \psi_{L})_{+}(x) - (\var - \psi_{L})_{+}(y) |  dx dy 
\end{split}
\end{equation}
where, in the above inequality, we have used the fact that 
$$ 
| (\var - \psi_{L})_{+}(x) - (\var - \psi_{L})_{+}(y) | \leq  \{  \chi_{\{  [\var - \psi_{L}](x) > 0     \}}   + \chi_{\{  [\var - \psi_{L}](y) > 0     \}}    \} | (\var - \psi_{L})_{+}(x) - (\var - \psi_{L})_{+}(y) |,  
$$ 
and the symmetry in $x$ and $y$.

Now, by Holder's inequality, and 
 using the elementary inequality $|\psi (y) - \psi (x)| < |y-x|$ , for any $x$, $y$ in $\mathbb{R}^{N}$, we can have the following estimation.
\begin{equation}\label{extra}
\begin{split}
& 2 \int \int_{|x-y| < 1} K(t,x,y) \chi_{\{  [\var - \psi_{L}](x) > 0     \}}  |\psi_{L} (x) - \psi_{L}(y)| |(\var - \psi_{L})_{+}(x) - (\var - \psi_{L})_{+}(y) |  dx dy \\
&  \leq a \cdot \int \int_{|x-y| < 1} K(t,x,y) \{ (\var - \psi_{L})_{+}(x) - (\var - \psi_{L})_{+}(y)  \} ^{2} dy dx \\
& + \frac{1}{a} \cdot  \int \int_{|x-y| < 1} K(t,x,y) |\psi (x) - \psi (y)|^{2} \cdot  \chi_{\{  [\var - \psi_{L}](x) > 0     \}} dy dx,
\end{split}
\end{equation}
in which the arbritary $a > 0$ will be chosen later.  Finally
\begin{equation*}
\begin{split}
& \int \int_{|x-y| < 1} K(t,x,y) |\psi (x) - \psi (y)|^{2} dy \cdot  \chi_{\{  [\var - \psi_{L}](x) > 0     \}} dx \\
& \leq \int_{\mathbb{R}^{N}} \int_{|x-y| < 1} \frac{2 \Lambda }{|x-y|^{N+s}} |y-x|^{2} dy \cdot  \chi_{\{  [\var - \psi_{L}](x) > 0     \}} dx
=C_s  \int_{\mathbb{R}^{N}} \chi_{\{  [\var - \psi_{L}](x) > 0     \}} dx .
\end{split}
\end{equation*}
Pulling this inequality in (\ref{5in1stDeGiorgi}) with $a=1/2$, and gathering it together with (\ref{3in1stDeGiorgi}), (\ref{4in1stDeGiorgi}), (\ref{5in1stDeGiorgi}), we can rewrite the energy inequality  as
\begin{equation}\label{8in1stDeGiorgi}
\begin{array}{l}
\qquad\ds{\frac{d}{dt} \int_{\mathbb{R}^{N}} [\var - \psi_{L}]_{+}^{2} dx + \frac{1}{2} B [ (\var - \psi_{L})_{+} , (\var - \psi_{L})_{+} ]}\\[0.3cm]
\ds{\leq C_{N, \Lambda,s} \{  \int_{\mathbb{R}^{N}} (\var - \psi_{L})_{+}(x) dx + \int_{\mathbb{R}^{N}} \chi_{\{  [\var - \psi_{L}](x) > 0     \}} dx     \},}
\end{array}
\end{equation}
where $C_{N, \Lambda,s}$ is some universal constant depending only on $N$ and $\Lambda$ and $s$. Next, in order to employ the Sobolev embedding theorem, we need to compare  $ B [ (\var - \psi_{L})_{+} , (\var - \psi_{L})_{+} ]   $ with $\| (\var - \psi_{L})_{+}  \|_{H^{\frac{s}{2}}(\mathbb{R}^{N})  }^{2} $ as follow.

\begin{equation*}
\begin{split}
&  \| (\var - \psi_{L})_{+}  \|_{H^{\frac{s}{2}}(\mathbb{R}^{N})  }^{2} \\
& = \int \int_{|x-y| \leq 2 } \frac{ \{ (\var - \psi_{L})_{+}(x) -  (\var - \psi_{L})_{+}(y) \}^{2}      }{|x-y|^{N+s}} 
+  \int \int_{|x-y| > 2 } \frac{ \{ (\var - \psi_{L})_{+}(x) -  (\var - \psi_{L})_{+}(y) \}^{2}      }{|x-y|^{N+s}}  \\
& \leq \Lambda \cdot B [ (\var - \psi_{L})_{+} , (\var - \psi_{L})_{+} ]   + 2 \int \int_{|x-y| > 2 } \frac{1}{|x-y|^{N+s}} \{ (\var - \psi_{L})_{+}^{2}(x) +  (\var - \psi_{L})_{+}^{2}(y) \} dx dy \\
& \leq   \Lambda \cdot B [ (\var - \psi_{L})_{+} , (\var - \psi_{L})_{+} ]  +C \int_{\mathbb{R}^{N}} (\var - \psi_{L})_{+}^{2} dx .
\end{split}
\end{equation*}
Hence, 
\begin{equation}\label{9in1stDeGiorgi}
\begin{split}
& \frac{d}{dt} \int_{\mathbb{R}^{N}} [\var - \psi_{L}]_{+}^{2} dx +  \frac{1}{\Lambda}\| (\var - \psi_{L})_{+}  \|_{H^\frac{s}{2}(\mathbb{R}^{N})  }^{2} \\
& \leq  C_{N, \Lambda ,s} \{  \int_{\mathbb{R}^{N}} (\var - \psi_{L})_{+} dx + \int_{\mathbb{R}^{N}} \chi_{\{  \var - \psi_{L} > 0     \}} dx    
+ | \int_{\mathbb{R}^{N}} (\var - \psi_{L})_{+}^{2} dx  \}.
\end{split}
\end{equation}
\vskip0.3cm \noindent{\bf Second step: Nonlinear  recurrence.}
From this energy inequality, we establish a nonlinear recurrence relation to the following sequence of truncated energy.

\begin{equation*}
U_{k} = \sup_{t \in [T_{k} , 0]} \int_{\mathbb{R}^{N}} (\var - \psi_{L_{k}})_{+}^{2} (t,x) dx + \int_{T_{k}}^{0}  \| (\var - \psi_{L_{k}})_{+} (t, \cdot )   \|_{H^{\frac{s}{2}}(\mathbb{R}^{N})  }^{2}\,dt
\end{equation*}

where, in the above expression, $T_{k} = -1 -\frac{1}{2^{k}}$, and $L_{k} = \frac{1}{2} (1-\frac{1}{2^{k}})$. Moreover, we will use the abbreviation $Q_{k} = [T_{k} , 0] \times \mathbb{R}^{N}$.
\vskip0.3cm

Now, let us consider two variables $\sigma$, $t$ which satisfies $T_{k-1} \leq \sigma \leq T_{k} \leq t \leq 0$. By taking the time integral over $[\sigma, t]$ in inequality \ref{9in1stDeGiorgi}, we yield

\begin{equation*}
\begin{split}
& \int_{\mathbb{R}^{N}} [\var - \psi_{L_{k}}]_{+}^{2} (t,x) dx +  \int_{\sigma}^{t}\| (\var - \psi_{L_{k}})_{+}  \|_{H^{\frac{s}{2}}(\mathbb{R}^{N})  }^{2} ds \\
& \leq  \int_{\mathbb{R}^{N}} [\var - \psi_{L_{k}}]_{+}^{2} (\sigma ,x) dx   +   C_{N, \Lambda,s }\{  \int_{\sigma }^{t} \int_{\mathbb{R}^{N}} (\var - \psi_{L_{k}})_{+}  +  \chi_{\{  \var - \psi_{L_{k}} > 0     \}}  +  (\var - \psi_{L_{k}})_{+}^{2} dx  ds \} .
\end{split}
\end{equation*}
 Next, by first taking the average over $\sigma \in [T_{k-1} , T_{k}]$, and then taking the sup over $t \in [T_{k} , 0]$ in the above inequality, we deduce from the above inequality that

\begin{equation}\label{10in1stDeGiorgi}
U_{k} \leq 2^{k} (1 + C_{N, s,\Lambda }) \{ \int_{Q_{k-1}}  (\var - \psi_{L_{k}})_{+}  +  \chi_{\{  \var - \psi_{L_{k}} > 0     \}}  +  (\var - \psi_{L_{k}})_{+}^{2} dx ds      \} .
\end{equation}
Using the Sobolev embedding theorem $H^{\frac{s}{2}} (\mathbb{R}^{N}) \subset L^{\frac{2N}{N-s}}(\mathbb{R}^{N})$ and interpolation we find
$$
\| (\var - \psi_{L_{k}})_{+}   \|_{L^{2(1+\frac{s}{N})}(Q_{k})} \leq C_{N} U_{k}^{\frac{1}{2}}.
$$
Using Tchebychev  inequality we get
\begin{align*}
\int_{Q_{k-1}}  (\var - \psi_{L_{k}})_{+} \leq  \int_{Q_{k-1}}  (\var - \psi_{L})_{+} \chi_{\{ \var - \psi_{L_{k-1}} > \frac{1}{2^{k+1}} \}} & \leq (2^{k+1})^{1+\frac{2s}{N}} \int_{Q_{k-1}} ( \var - \psi_{L_{k-1}} )_{+}^{2(1+\frac{s}{N})} \\
& \leq (2^{k+1})^{1+\frac{2s}{N}} C_{N}^{2(1+\frac{s}{N})} U_{k-1}^{1+\frac{s}{N}} .
\end{align*}

\begin{align*}
\int_{Q_{k-1}} \chi_{\{  \var - \psi_{L_{k}} > 0     \}} \leq (2^{k+1})^{2(1+\frac{s}{N})} \int_{Q_{k-1}} ( \var - \psi_{L_{k-1}} )_{+}^{2(1+\frac{s}{N})}
\leq (2^{k+1})^{2(1+\frac{s}{N})} C_{N}^{2(1+\frac{s}{N})} U_{k-1}^{1+\frac{s}{N}} .
\end{align*}

\begin{align*}
\int_{Q_{k-1}}  (\var - \psi_{L_{k}})_{+}^{2} \leq \int_{Q_{k-1}}  (\var - \psi_{L})_{+}^{2} \chi_{\{ \var - \psi_{L_{k-1}} > \frac{1}{2^{k+1}} \}} & \leq 
(2^{k+1})^{\frac{2s}{N}} \int_{Q_{k-1}} ( \var - \psi_{L_{k-1}} )_{+}^{2(1+\frac{s}{N})} \\
& \leq (2^{k+1})^{\frac{2s}{N}} C_{N}^{2(1+\frac{s}{N})} U_{k-1}^{1+\frac{s}{N}}  .
\end{align*}

The above three inequalities, together with inequality (\ref{10in1stDeGiorgi}),  give

\begin{equation}\label{recurrence}
U_{k} \leq \{\overline{C}_{N,\Lambda,s}\}^{k} U_{k-1}^{1+\frac{s}{N} } , \forall k \geq 0 ,
\end{equation}

for some universal constant $\overline{C}_{N,\Lambda,s}$ depending only on $N$, $s$, and $\Lambda$. Due to the nonlinear recurrence relation (\ref{recurrence}) for $U_{k}$, we know there exists some sufficiently small universal constant $\epsilon_{0} = \epsilon_{0} (\overline{C}_{N,\Lambda,s}   )$, depending only on $\overline{C}_{N,\Lambda,s}$, such that the following implication is valid.
\vskip0.3cm
\noindent  If $U_{1} \leq \epsilon_{0}$, then it follows that $\lim_{k\rightarrow \infty } U_{k} = 0$. 
\vskip0.3cm
\noindent Equation (\ref{10in1stDeGiorgi}) with Tchebychev inequality gives that
$$
U_1\leq C \int_{-2}^0\int_{\R^N}|\var-\psi|^2\,dx\,dt,
$$
and $U_k$ converges to 0 implies that 
$$
\var\leq \psi+\frac{1}{2}\qquad t\in[-1,0]\times\R^N.
$$

\end{proof}

We have the following corollary of Lemma \ref{firstDeGiorgilemma}. It shows that any solutions is indeed bounded for $t>0$.
\begin{cor}\label{cor1}
Any solution to (\ref{star}) with initial value in $L^2(\R^N)$ is uniformly bounded on $(t_0,\infty)\times\R^N$ for any $0<t_0<2$. Indeed:
$$
\sup_{t>t_0,x\in\R^N}|\var(t,x)|\leq \frac{\|\var^0\|_{L^2}}{2\sqrt{\eps_0}(t_0/2)^{(N/s+1)/2}}.
$$
\end{cor}

\begin{proof}
Fix $0<t_0<2$ and $x_0\in\R^N$, for any $t>-2$, $x\in\R^N$ we consider 
$$
\bar{\var}(t,x)=\frac{(t_0/2)^{(N/s+1)/2}\sqrt{\eps_0}}{\|\var^0\|_{L^2}}\var(t_0+t(t_0/2),x_0+x(t_0/2)^{1/s}).
$$
The function $\bar{\var}$ still verifies the equation (\ref{main}) with an other kernel verifying Hypothesis (\ref{diese}) with the same constant $\Lambda$.
From the decreasing of energy, $\bar{\var}$ verifies the assumptions of  Lemma \ref{firstDeGiorgilemma}. Hence $\bar{\var}(0,0)\leq 1/2$.
Working with $-\bar{\var}$ gives that $-\bar{\var}(0,0)\leq 1/2$ too.
\end{proof}

We define $\psi_1(x)=(|x|^{s/4}-1)_+$.
We can rewrite the main lemma of this section in the following way. It will be useful for the next section. 
\begin{cor}\label{cor2}
Let $\Lambda$ be the given constant in condition (\ref{diese}). Then, there exists a constant $\delta \in (0, 1)$,  depending only on $N$, $s$, and $\Lambda$,  such that for any solution $\var : [-2, 0] \times \mathbb{R}^{N} \rightarrow \mathbb{R}$ to (\ref{main}) satisfying 
$$
\var (t,x)  \leq  1+\psi_1(x)\qquad \mathrm{on}\   [-2 , 0]\times \mathbb{R}^{N}, 
$$
and
$$
|\{\var>0\}\cap ([-2,0]\times B_{2})|\leq\delta,
$$
 we have 
  $$
  \var (t,x) \leq  \frac{1}{2} , \qquad (t,x) \in [-1,0]\times B_1.
  $$ 
\end{cor}

\begin{proof}
Consider $R\geq2$ such that $1+\psi_1(R)\leq \psi(R)$. 
 Note that $R$ depends only on $s$.
 For any $(t_0,x_0)\in [-1,0]\times B_{1}$ we introduce $w_R$ defined on $(-2,0)\times\R^N$ by
 $$
 w_R(t,x)=w((t-t_0)/R^s,(x-x_0)/R).
 $$
 Note that $w_R$  verifies the equation (\ref{main}) with an other kernel verifying Hypothesis (\ref{diese}) with the same constant $\Lambda$.
 Since $\psi_1$ increase with respect to $|x|$, for $|x|>1$ we have
 $$
 w_R(t,x)\leq 1+\psi_1\left(\frac{x-x_0}{R}\right)\leq 1+\psi_1\left(\frac{1+|x|}{R}\right)\leq 1+\psi_1\left((2/R)|x|\right)\leq 1+\psi_1(x).
 $$
 So, from the definition of $R$, for $|x|\geq R$ we have $w_R(t,x)\leq \psi(x)$.
 Hence,
from the hypothesis we have
\begin{eqnarray*}
&&\int_{-2}^0\int_{\R^N}(w_R(t,x)-\psi(x))_+^2\,dx\,dt=\int_{-2}^0\int_{B_{R}}(\var_R(t,x)-\psi(x))_+^2\,dx\,dt \\
&&\qquad\qquad\leq R^{N+s}\int_{-2}^0\int_{B_2}(\var(t,x))_+^2\,dx\,dt\leq R^{N+s}(1+\psi_1(2))^2\delta.
\end{eqnarray*}
So, choosing $\delta=R^{-(N+s)}(1+\psi_1(2))^{-2}\eps_0$ gives that $w(t_0,x_0)\leq 1/2$ for $(t_0,x_0)\in(-1,0)\times B_1$.

\end{proof}

\section{The second De-Giorgi's lemma}

This section is dedicated to a lemma of local decrease of the oscillation of a solution to Equation (\ref{main}). 
%
We define the following function
\begin{equation}\label{F}
\pha(x)=\sup(-1,\inf(0,|x|^2-9)).
\end{equation}
Note that $\pha$ is Lipschitz, compactly supported in $B_{3}$, and  equal to  -1 in $B_2$.

For  $\lambda<1/3$, we define
\begin{eqnarray*}
\psi_{\lambda}(x)&=&0,\qquad {\mathrm{if}} \ \  |x|\leq \frac{1}{\lambda^{4/s}},\\
&=& ((|x|-1/\lambda^{4/s})^{s/4}-1)_+,\qquad {\mathrm{if}} \ \  |x|\geq\frac{1}{\lambda^{4/s}} . 
\end{eqnarray*}

The normalized lemma will involve three consecutive cut-offs:
$$\varphi_0 = 1+\psi_\lambda+\pha,$$
$$\varphi_1 = 1+\psi_\lambda+\lambda\pha,$$
$$\varphi_2 = 1+\psi_\lambda+\lambda^2\pha.$$

We prove the following lemma:
\begin{lemma}\label{lemma2}
Let $\Lambda$ be the given constant in condition (\ref{diese}) and  $\delta$ the constant defined in Corollary \ref{cor2}. Then, there exists $\mu>0$, $\gamma>0$, and $\lambda \in (0,1) $,  depending only on $N$,  $\Lambda$, and $s$, such that for any solution $\var : [-3, 0] \times \mathbb{R}^{N} \rightarrow \mathbb{R}$ to (\ref{main}) satisfying 
$$
\var (t,x)  \leq 1 + \psi_\lambda(x)\qquad \mathrm{on}  [-3 , 0]\times \mathbb{R}^{N}, 
$$
$$
|\{\var<\varphi_0\}\cap((-3,-2)\times B_{1})|\geq \mu,
$$
 then  we have either
 $$
|\{\var>\varphi_2\}\cap((-2,0)\times\R^N)|\leq \delta,
$$
 or
  $$
  |\{\varphi_0<\var <\varphi_2\}\cap((-3,0)\times\R^N)|\geq \gamma.
  $$ 
 \end{lemma}

The lemma says that in going from the $\varphi_0$ cut off to the $\varphi_2$
cut off, i.e., from the set 
$\{ \var> \varphi_0\}$
to 
$\{ \var>\varphi_2\}$
``some mass'' is lost, i.e., if $|\{ \var>\varphi_2\}|$ is not yet subcritical 
(i.e., $\le \delta$)
then 
$$|\{ \var>\varphi_2\}| \le |\{ \var>\varphi_0\}| -\gamma.$$ 

\begin{proof} In all the proof, we denote by $C$ constants which depend only on $s$, $N$ and $\Lambda$, but which can change from a line to another.
We may fix any $0<\mu<1/8$. We will fix $\delta$ smaller than the one in Corollary \ref{cor2} and such that the term $C\delta$  in (\ref{Cdelta})
is smaller than $1/4$. The task consists now in showing that for $0<\lambda<1/3$ small enough, there exists a $\gamma>0$ for which the lemma holds.  
The constraints on $\lambda$ are (\ref{lambda1}), (\ref{lambda2}), (\ref{lambda3}), and (\ref{lambda4}).
We split the proof into several steps.

\noindent{\bf First step: The energy inequality.}
We start again with  the energy inequality (\ref{energie}), but  use better the ``good''  term
$$B((\var-\varphi)_+, (\var-\varphi)_-)
= \iint_{\R^{2N}} (\var-\varphi)_+ (x)K(t,x,y)(\var-\varphi)_{\mathrm{neg}} (y))\,dx\,dy$$
that we just neglected before.
\medskip

We have, for  $\varphi_1$ the intermediate cut off (see (\ref{energie})):
\begin{equation*}
\begin{split}
&\int ((\var-\varphi_1)_+)^2 \,dx \Big|_{T_1}^{T_2} + \int_{T_1}^{T_2} 
B( (\var-\varphi_1)_+,(\var-\varphi_1)_+) \,dt=\\
\noalign{\vskip18pt}
&\qquad 
- \int_{T_1}^{T_2} B((\var-\varphi_1)_+,\varphi_1) \,dt
- \int_{T_1}^{T_2} B((\var-\varphi_1)_+, (\var-\varphi_1)_-)\,dt.
\end{split}
\end{equation*}
The remainder term can be controlled in the following way.
$$B((\var-\varphi_1)_+,\varphi_1) \le \frac12 B((\var-\varphi_1)_+,(\var-\varphi_1)_+) 
+ 2\mkern-6mu
\iint [\varphi_1 (x) -\varphi_1(y)] K (x,y)
[\varphi_1 (x) -\varphi_1(y)] [\chi_{B_{3}} (x)].$$
The first term $\frac12 B((\var-\varphi_1)_+, (\var-\varphi_1)_+)$ is absorbed on the left.  
The second one is smaller than
$$
4\lambda^2\iint [F (x) -F(y)] K (x,y)
[F(x) -F(y)] +4\iint [\psi_\lambda (x) -\psi_\lambda (y)] K (x,y)
[\psi_\lambda (x)-\psi_\lambda (y)] [\chi_{B_{3}} (x)],
$$
which is smaller that $C\lambda^2$. This is obvious for the first term since $F$ is Lipschitz and compactly supported.  Since $\psi_\lambda(x)=0$ for $|x|<3$, the second term is equal to 
\begin{eqnarray*}
&&\qquad\qquad 4\iint\psi_\lambda(y)^2 [\chi_{B_{3}} (x)]K(t,x,y)\,dx\,dy\\
&& \leq 4(1-s/2)\Lambda|B_3|\int_{\{|y|>1/\lambda^{4/s}\}} \frac{(((|y|-1/\lambda^{4/s})^{s/4}-1)^2_+}{(|y|-3)^{N+s}}\,dy\\
&&\leq 4(1-s/2)\Lambda |B_3|\lambda^2\int_{\{|z|>1\}} \frac{(((|z|-1)^{s/4}-\lambda)^2_+}{(|z|-3\lambda^{4/s})^{N+s}}\,dz\\
&&\leq 4(1-s/2)\Lambda |B_3|\lambda^2\int_{\{|z|>1\}} \frac{(((|z|-1)^{s/4})^2_+}{(|z|-1/3)^{N+s}}\,dz\\
&&\leq C\lambda^2,
\end{eqnarray*}
since $\lambda<1/3$.

This leaves us with the inequality 
\begin{equation*}
\begin{split}
&\int (\var-\varphi_1)_+^2 dx \Big|_{T_1}^{T_2} + \frac12 \int_{T_1}^{T_2} 
B((\var-\varphi_1)_+, (\var-\varphi_1)_+) \,dt\\
\noalign{\vskip18pt}
&+  \int_{T_1}^{T_2} \int_{\R^{2N}} (\var-\varphi_1)_+ (x) K(x,y) 
(\var-\varphi_1)_{\mathrm{neg}} (y)\,dx\,dy\,dt\le C\ \lambda^2 (T_2-T_1).
\end{split}
\end{equation*}
In particular, since the second and third terms are positive, 
we get that for $-3< T_1<T_2<0$:

 $$\ds H(t) = \int_{\R^N} (\var-\varphi_1)_+^2 (t,x) dx$$

satisfies 
$$H' (t) \le C\ \lambda^2,$$
and 
\begin{equation}\label{goodterm}
\ds \int_{T_1}^{T_2} \int_{\R^{2N}} (\var-\varphi_1)_+ (x) K(x,y) 
(\var-\varphi_1)_{\mathrm{neg}} (y)\,dx\,dy\,dt\le C\ \lambda^2 [T_2-T_1].
\end{equation}
Note that, up to now, those estimates hold for any $0<\lambda<1/3$.
\vskip0.3cm
\noindent{\bf Second step:  An estimate on those time slices  where the ``good'' extra term helps.}
Remember that  $\mu<1/8$ is fixed from the beginning of the proof. 
From our hypothesis
$$|\{ \var< \varphi_0\}\cap((-3,-2)\times B_1)| \ge \mu,$$
the set of times $\Sigma$ in $(-3,-2)$ for which $|\{ \var(\cdot,T)<\varphi_0\}\cap B_1| \ge \mu/4$ 
has at least measure $\mu/(2|B_{1}|)$.


We estimate now that except for a few of those time slices, 
$\ds\int_{\R^N} (\var-\varphi_1)_+^2\,dx$ is very tiny:
\bigskip

Since \quad $\ds\inf_{|x-y| \le3} K(t,x,y) \geq C\Lambda^{-1}$\quad we have that
\begin{equation*}
\begin{split}
C\,\lambda^2 
&\ge \int_{-3}^{-2} B((\var-\varphi_1)_+,(\var-\varphi_1)_-)dt 
\ge C\,\Lambda^{-1} \frac{\mu}{8} \int_{ \Sigma} \int_{\R^N}(\var-\varphi_1)_+ \,dx\,dt\\
\noalign{\vskip12pt}
&\ge C\,\Lambda^{-1}\frac{\mu}{8\lambda} \int_{\Sigma}\int_{\R^N} 
(\var-\varphi_1)_+^2\,dx\,dt
\end{split}
\end{equation*}
since $(\var-\varphi_1)_+ \le \lambda$.


In other words
$$\int_{\Sigma}\int_{\R^N} [(\var-\varphi_1)_+(x)]^2\,dx\,dt
\le \overline{C} \ \frac{\lambda^3}{\mu} 
\le  \lambda^{3-1/8}$$
if $\lambda$ is small enough such that
\begin{equation}\label{lambda1}
\lambda\leq \left(\frac{\mu}{\overline{C}}\right)^8.
\end{equation}
\bigskip

In particular, from Tchebychev's inequality:
\begin{equation}\label{eq_10}
\int (\var-\varphi_1)_+^2 (t,x)\,dx \le  \lambda^{3-\frac14} 
\end{equation}
for all $t\in \Sigma$, except for a very small subset $F$ of $t$'s of measure smaller than $\lambda^{1/8}$. We need it
still much smaller than $\mu \sim |\Sigma|$. indeed, if $\lambda$ is small enough such that
\begin{equation}\label{lambda2}
\lambda\leq \left(\frac{\mu}{4|B_1|}\right)^8,
\end{equation}
then,  (\ref{eq_10})
 holds on a set of $t$s  in $[-3,-2]$ of measure bigger than $\mu/(4|B_1|)$.
\vskip0.3cm
\noindent{\bf Third step.  In search of an intermediate set,
 where $\var$ is between $\varphi_0$ and 
$\varphi_2$.}
Let us go now to $(\var-\varphi_2)_+$.

Assume that for at least one time $T_0>-2$,
$$|\{ x\ | \ (\var-\varphi_2)_+ (T_0,x)>0\}| > \delta/2 ,$$
i.e., goes over critical for the first lemma and let's go backwards in time 
until we reach a slice of time $T_1 \in \Sigma$, where 
$$\int_{\R^N} (\var-\varphi_1)_+^2 (T_1,x)\,dx\le  \lambda^{3-\frac14}.$$


At $T_0$, for the intermediate cut off, $\varphi_1$, we have 
\begin{equation}\label{exemple}
\begin{array}{l}
\ds{\int_{\R^N} (\var-\varphi_1)_+^2(T_0,x)\,dx \ge \int (\varphi_1 -\varphi_2)^2 \chi_{\{(\var-\varphi_2)_+>0\}}}\\[0.3cm]
\qquad
\ds{\ge \int (\lambda-\lambda^2)^2F^2(x) \chi_{\{(\var-\varphi_2)_+>0\}} \geq C_F\frac{\lambda^2}{4} \delta^3,}
\end{array}
\end{equation}
where the constant $C_F$ depends only on the fixed function $F$. Indeed we have
 $\lambda<1/2$, and $F$ is increasing with respect to $|x|$ and smaller than $-C(3-|x|)$ for $|x|<3$  closed to $3$.  Hence, the integral is minimum
 when all the mass $\{(\var-\varphi_2)_+>0\}$ is concentrated on $3-C\delta<|x|<3$.
\vskip0.3cm
Now,  at $T_1$, 
$$\int_{\R^N} (\var-\varphi_1)_+^2(T_1,x)\,dx \le   \lambda^{3-\frac14}\ .$$


Thus, for $\lambda$ small enough such that
\begin{equation}\label{lambda3}
\lambda^{1-1/4}\leq C_F\frac{\delta^3}{64},
\end{equation}
in going from $T_0$ backwards to $T_1$, $H(t) = \int_{\R^N} (\var-\varphi_1)_+^2(t,x)\,dx$ 
has crossed a range between two multiples of $\delta^3\lambda^2$, 
from say $\lambda^2 \frac{\delta^3}8$, to 
$\lambda^2 \frac{\delta^3}{16}$. 
  Since $H' (t) \le C\ \lambda^2$, in order to do so 
it needed in a range of times $D$, of at 
least length $\sim\delta^3$, where
$$
D=\{t\in(T_1,T_0): \ \lambda^{3-1/4}<H(t)<C_F\frac{\lambda^2}{4}\delta^3\}.
$$ 

We want to show that in this range, we pick up an intermediate set, of nontrivial measure, where $(\var-{\varphi_0})_+ >0$ and $(\var-\varphi_2)_+=0$,
implying that the measure 
$$\A_2 = |\{ (\var-\varphi_2)_+ >0\}\cap\{t\in (-3,0)\}|$$
effectively decreases some fixed amount from 
$$\A_0 = |\{ \var-{\varphi_0})_+ >0\}\cap\{t\in(-3,0)\}|\ .$$


In these range of times $D$, given the gap between $\varphi_1$ and $\varphi_2$
\begin{equation}\label{Cdelta}
|\{ (\var-\varphi_2)_+ > 0\}| \le C\ \delta
\end{equation}
(if not 
$\int (\var-\varphi_1)_+^2 > \delta^3\lambda^2$ as in the computation of (\ref{exemple})). As said in the beginning of the proof, we may consider a $\delta$
such that $C\delta<1/4$. 
Moreover, those times of $D$ for which 
$$|\{(\var-\varphi_0)_+ \leq 0\}\cap B_2| \ge \mu$$
are in an exceptional subset $\mathcal{F}$ of very small size.
Indeed
\begin{eqnarray*}
&&C\lambda^2\geq\int_{-3}^0\int_{\R^N}(w-\varphi_1)_+K(t,x,y)(w-\varphi_1)_{\mathrm{neg}}\\
&&\qquad \geq C\mu \int_{\mathcal{F} }\int_{B_3}(w-\varphi_1)_+\,dx\,dt\geq \frac{C\mu}{\lambda}\int_{\mathcal{F}}\int_{\R^N}(w-\varphi_1)_+^2\,dx\,dt\\
&&\qquad \geq \frac{C\mu|\mathcal{F}|(\lambda^2\delta^3/16)}{\lambda}.
\end{eqnarray*}
Hence
$$
|\mathcal{F}|\leq C \frac{\lambda}{\mu\delta^3}.
$$
And so, for $\lambda$ small enough such that 
\begin{equation}\label{lambda4}
\lambda\leq \mu\delta^3 |D|/(2C),
\end{equation}
we have 
$$
|\mathcal{F}|\leq\frac{|D|}{2}.
$$
Note that the constraint (\ref{lambda4}) can be expressed depending only on $s$, $N$, $\Lambda$, $\delta$, and, $\mu$, since $|D|<C\delta^3$.
  For these times in $D$ not in $\mathcal{F}$, we have: 
$$A(t) = |\{ \varphi_0 \le \var(t,\cdot)\le \varphi_2\}| \ge 1/2.$$


That is
\begin{eqnarray*}
&&|\{\varphi_0< w< \varphi_2\}\cap((-3,0)\times\R^N)|\geq \int_{-3}^0 A(t)\,dt\\
&&\qquad\qquad \geq \int_{D\setminus \mathcal{F}} A(t)\,dt\geq  \frac{|D|}{4}\geq C\delta^3.
\end{eqnarray*}
\end{proof}

\section{Proof of the $C^\alpha$ regularity}

we are now ready to show the following oscillation lemma.
First, for $\lambda$ as in the previous section, we define for any $\eps>0$
\begin{eqnarray*}
\psi_{\eps,\lambda}(x)&=&0,\qquad {\mathrm{if}} \ \  |x|\leq \frac{1}{\lambda^{4/s}},\\
&=& ((|x|-1/\lambda^{4/s})^{\eps}-1)_+,\qquad {\mathrm{if}} \ \  |x|\geq\frac{1}{\lambda^{4/s}} . 
\end{eqnarray*}

\begin{lemma}\label{lastlemma}
there exists $\eps>0$ and $\lambda^*$ such that for any solution to (\ref{main}) in $[-3,0]\times \R^N$ such that
$$
-1-\psi_{\eps,\lambda}\leq\var\leq 1+\psi_{\eps,\lambda},
$$ 
we have
$$
\sup_{[-1,0]\times B_1}\var -\inf_{[-1,0]\times B_1}\var\leq 2-\lambda^*.
$$
\end{lemma}
\begin{proof}
We may assume that 
$$
|\{w<\varphi_0\}\cap((-3,-2)\times B_1)|>\mu.
$$
Otherwise this is verified by $-w$, and we may work on this function.

Consider $k_0=|(-3,0)\times B_3|/\gamma$. Then we fix $\eps$ small enough such that 
$$
\frac{(|x|^\eps-1)_+}{\lambda^{2k_0}}\leq (|x|^{s/4}-1)_+,
$$
for all $x$. We may take $\eps=(s/4)\lambda^{2k_0}$ for instance.
For $k\leq k_0$, we consider the sequence
$$
\var_{k+1}= \frac1{\lambda^2} (\var_k-(1-\lambda^2)),\qquad \var_1=\var.
$$
By induction, we have that
$$
(w_k)_+(t,x)\leq 1+\frac{1}{\lambda^{2k}}\psi_{\eps,\lambda}(x), \qquad t\in(-3,0), x\in \R^N.
$$
So, for $k\leq k_0$ we have $w_k\leq 1+\psi_\lambda$.
By construction $|\{w_k< \varphi_0\}\cap(-3,-2)\times B_1|$ is increasing, so bigger than $\mu$ for any $k$.
Hence, we can apply Lemma \ref{lemma2} on $w_k$. 
As long as $|\{w_k>\varphi_2\}\cap ((-2,0)\times \R^N)|\geq \delta$, we have
$$
|\{w_{k+1}> \varphi_0\}|=|\{w_{k+1}>\varphi_2\}|
+|\{\varphi_0< w_{k+1}< \varphi_2\}|,
$$
and
 \begin{eqnarray*}
&& |\{w_{k+1}> \varphi_2\}|\leq |\{w_{k+1}>\varphi_0\}|-\gamma\\
&&\qquad\qquad\leq  |\{w_{k}>\varphi_2\}|-\gamma\leq |(-3,0)\times B_3|-k\gamma. 
 \end{eqnarray*}
 This cannot be true up to $k_0$. 
 So there exists $k\leq k_0$ such that 
$$
|\{\var_k>\varphi_2\}\cap((-2,0)\times \R^N)|\leq \delta.
$$
We can then apply the first De Giorgi lemma on $\var_{k+1}$. Indeed
$$
w_{k+1}\leq 1+\psi_\lambda\leq 1+\psi_1, \qquad \mathrm{on} \ (-3,0)\times\R^N,
$$
and
\begin{eqnarray*}
&&|\{w_{k+1}>0\}\cap((-2,0)\times B_2)|\leq |\{w_{k+1}>\varphi_0\}\cap((-2,0)\times B_2)|\\
&&\qquad\qquad \leq  |\{w_{k}>\varphi_2\}\cap((-2,0)\times \R^N)|\leq\delta.
\end{eqnarray*}
Hence, from Corollary \ref{cor2}, we have
$$
w_{k+1}\leq 1/2, \qquad\mathrm{on} (-1,0)\times B_1.
$$
This gives the result with 
$$
\lambda^*=\frac{\lambda^{2k_0}}{2}.
$$
\end{proof}

\medskip
The $C^\alpha$ regularity follows in a standart way. 
\begin{proof}
For any $(t_0,x_0)\in (0,\infty)\times \R^N$, consider first $K_0=\inf(1, t_0/4)^{1/s}$, and
$$
w_0(t,x)=w(t_0+K_0^st,x_0+K_0x).
$$
This function still verifies  an equation  of the type of (\ref{main})  in $(-4,0)\times\R^N$ with a kernel $K$ verifying (\ref{diese}). 
From Corollary \ref{cor1}, It is bounded on $(-3,0)\times\R^N$.
Consider $K<1$ such that 
$$
\frac{1}{1-(\lambda^*/2)}\psi_{\lambda,\eps}(Kx)\leq \psi_{\lambda, \eps}(x),\qquad \mathrm{for} \ |x|\geq1/K.
$$
The coefficient $K$ depends only on $\lambda$, $\lambda^*$ and $\eps$.
Then we define by induction:
\begin{eqnarray*}
&&w_1(t,x)=\frac{w_0(t,x)}{\|w_0\|_{L^\infty}}, \qquad (t,x)\in (-3,0)\times\R^N,\\
&&w_{k+1}(t,x)=\frac{1}{1-\lambda^*/4}\left(w_{k}(K^s t,Kx)-\bar{w}_k\right),\qquad (t,x)\in (-3,0)\times\R^N,
\end{eqnarray*}
where 
$$
\bar{w}_k=\frac{1}{|B_1|}\int_{-1}^0\int_{B_1}w_k(t,x)\,dx\,dt.
$$
By construction, $w_k$ verifies the hypothesis of Lemma \ref{lastlemma} for any $k$. 
Hence:
$$
\sup_{(t_0+(-K^{ks},0))\times(x_0+B_{K^k})}w-\inf_{(t_0+(-K^{ks},0))\times(x_0+B_{K^k})}w\leq C(1-\lambda^*/4)^k.
$$
So, $w$ is $C^\alpha$ with
$$
\alpha=\frac{\ln(1-\lambda^*/4)}{\ln(K^s)}.
$$

\end{proof}

\bibliography{bigproject.bib}

\end{document}